\documentclass{article}

\usepackage{amsmath,amssymb,amsthm,mathtools}
\usepackage{enumitem}

\newtheorem{theorem}{Theorem}[section]
\newtheorem{proposition}[theorem]{Proposition}
\newtheorem{lemma}[theorem]{Lemma}

\theoremstyle{definition}

\theoremstyle{remark}

\title{There Are Exactly $2^\mathfrak{c}$ Simultaneously Bireflective and Bicoreflective Subconstructs of $\mathbf{FTS}$}
\author{Gao Zhang \\
  School of Mathematical Sciences \\
  Jiangsu Second Normal University \\
  Nanjing 211222, China \\
\texttt{gaozhang0810@hotmail.com}}
\date{}
\begin{document}

\maketitle

\begin{abstract}
  We prove that there are exactly $2^\mathfrak{c}$ simultaneously bireflective and bicoreflective subconstructs of $\mathbf{FTS}$. By the Lowen--Wuyts
  bijection, these correspond to the total fuzzy topologies contained in $M$, the fuzzy topology consisting of all nondecreasing left-continuous self-maps of the unit interval. The upper bound follows from a cardinality estimate. The matching lower bound is obtained by encoding $2^{\mathfrak c}$ ultrafilters on $\mathbb N$ into total fuzzy topologies contained in $M$.
\end{abstract}


\section{Introduction}

The construct $\mathbf{FTS}$ of fuzzy topological spaces admits
a rich family of full subconstructs that are both
bireflective and bicoreflective. Such a subconstruct is closed under
the formation of initial and final structures in $\mathbf{FTS}$ and
therefore provides a self-contained theory of fuzzy topological
spaces.

Throughout this paper, all subconstructs of $\mathbf{FTS}$ are understood to be full and replete and equipped with the restriction $U|_{\mathcal A}$ of the forgetful functor
$U\colon\mathbf{FTS}\to\mathbf{Set}$.

Many simultaneously bireflective and bicoreflective subconstructs of $\mathbf{FTS}$ have been studied in detail \cite{morsi1995a,morsiTwoTypesStable1995,lowenFuzzyNeighborhoodSpaces1982,lowen1995,wuytsReflectorsCoreflectorsCategory1988}, including the construct of fuzzy neighborhood spaces $\mathbf{FNS}$, the construct of fuzzy neighborhood base spaces $\mathbf{FNBS}$, and the full image of the category $\mathbf{TOP}$ of ordinary topological spaces under the Lowen functor.

Lowen and Wuyts \cite{lowen1995} showed that there is a bijection
\[
  \tau\longmapsto\operatorname{Stab}(\tau)
\]
from the total fuzzy topologies contained in $M$ onto the simultaneously bireflective and bicoreflective subconstructs of $\mathbf{FTS}$, where $M$ consists
of all nondecreasing left-continuous self-maps of the unit interval.
Using this bijection, they constructed $\mathfrak{c}$
such subconstructs contained in $\mathbf{FNS}$ and another
$\mathfrak{c}$ contained in $\mathbf{FNBS}$, where \(\mathfrak{c}=|\mathbb{R}|\) denotes the cardinality of the continuum.
Li and Luo \cite{li2003} showed that the number of simultaneously bireflective and bicoreflective
subconstructs of $\mathbf{FTS}$ contained in $\mathbf{FNS}$ is exactly
$\mathfrak{c}$. For the construct $\mathbf{Ap}$ of approach spaces, Lowen and Wuyts
\cite{lowen2000,lowen2011} showed there are exactly $\mathfrak{c}$
such subconstructs.

This raises the natural question of how many simultaneously bireflective and bicoreflective subconstructs of $\mathbf{FTS}$ there are in total.

In the present paper, we prove the following result.
\begin{theorem}\label{thm:main}
  There are exactly   \(
    2^{\mathfrak{c}}
  \) total fuzzy topologies contained in $M$, and hence, by the Lowen--Wuyts bijection, exactly   \(
  2^{\mathfrak{c}}\) simultaneously bireflective and bicoreflective subconstructs of $\mathbf{FTS}$.
\end{theorem}

The upper bound follows from a cardinality estimate. To establish the matching lower bound, we construct $2^{\mathfrak{c}}$ distinct total fuzzy topologies contained in $M$.

\section{Preliminaries}
Let $\mathbf{I}=[0,1]$, and write $\underline{r}$ for the constant map with value
$r\in\mathbf{I}$. A fuzzy topology \cite{lowen1976} on a set $X$ is a family of functions $\tau\subseteq \mathbf{I}^X$ satisfying the following conditions:
\begin{enumerate}[label=(\arabic*)]
  \item \(\underline{r}\in\tau\) for every \(r\in\mathbf{I}\);
  \item \(\displaystyle\bigvee_{j\in J}\mu_j\in\tau\) for every
    family \((\mu_j)_{j\in J}\subseteq\tau\);
  \item \(\mu\wedge\nu \in\tau\) for all \(\mu,\nu\in\tau\).
\end{enumerate}The pair $(X,\tau)$ is called a fuzzy topological space. The intersection of any family of fuzzy topologies on $X$ is again a fuzzy topology. Given $F\subseteq \mathbf{I}^X$, the fuzzy topology generated by $F$ is the intersection of all fuzzy topologies on $X$ containing $F$.

A function $f\colon (X,\tau_1)\to(Y,\tau_2)$ between fuzzy topological spaces is called continuous if $\mu\circ f\in\tau_1$ for every $\mu\in\tau_2$. Fuzzy topological spaces and continuous functions form a category, denoted by $\mathbf{FTS}$.

Following Lowen \cite{lowen1976}, the Lowen functor
\[
  \omega\colon\mathbf{TOP}\longrightarrow\mathbf{FTS}
\]
assigns to each topological space \((X,\mathcal{T})\) the fuzzy
topological space \((X,\omega(\mathcal{T}))\), where
\[
  \omega(\mathcal{T})
  =
  \{\mu\colon X\to\mathbf I\mid
  \mu\text{ is lower semicontinuous}\}.
\]On morphisms, \(\omega\) acts as the identity on the underlying maps.

Let
\[
  \mathcal{T}_r
  =
  \{\mathbf I\}
  \cup
  \{(a,1]\mid a\in\mathbf I\}.
\]Then, it is easy to check that
\[
  \omega(\mathcal{T}_r)=M
  =
  \{f\colon\mathbf I\to\mathbf I\mid
  f\text{ is nondecreasing and left-continuous}\}.
\]

\begin{lemma}\label{lem:left}
  Every nondecreasing left-continuous map \(f\colon\mathbf I\to\mathbf I\) preserves nonempty suprema. That is,
  \[f(\bigvee S)=\bigvee_{s\in S}f(s)\]for every nonempty $S\subseteq\mathbf{I}$.
\end{lemma}
\begin{proof}
  Let $a=\bigvee S$. If $a\in S$, the conclusion is immediate. Otherwise,
  $S$ is cofinal in $[0,a)$, and hence, by monotonicity of $f$,
  \[
    \bigvee_{x<a}f(x)=\bigvee_{s\in S}f(s).
  \]
  Since $f$ is left-continuous,
  \[
    f(a)=\bigvee_{x<a}f(x)=\bigvee_{s\in S}f(s).
  \]
\end{proof}

A fuzzy topology $\tau$ on $\mathbf{I}$ is called total \cite{lowen1995} if
\[
  \mathrm{id}_{\mathbf{I}}\in\tau
  \qquad\text{and}\qquad
  f,g\in\tau\Longrightarrow f\circ g\in\tau.
\]It is easy to check that $M$ is a total fuzzy topology on $\mathbf{I}$.

We regard $\mathbf{FTS}$ as a construct \cite{adamekAbstractConcreteCategories1990}, that is, as a concrete
category over $\mathbf{Set}$ via the forgetful functor
\[
  U\colon\mathbf{FTS}\longrightarrow\mathbf{Set},
  \qquad
  U(X,\tau)=X.
\]

A subconstruct $\mathcal A$ of $\mathbf{FTS}$ is a subcategory
$\mathcal A$ of $\mathbf{FTS}$ equipped with the restricted
forgetful functor $U|_{\mathcal A}$.

Let $\mathcal A$ be such a subconstruct, and let
\[
  J\colon\mathcal A\hookrightarrow\mathbf{FTS}
\]
be the inclusion functor. The subconstruct $\mathcal A$ is reflective
if $J$ has a left adjoint $R$, and coreflective if $J$ has a right
adjoint $C$. Denote the corresponding unit and counit by
\[
  \eta_X\colon X\longrightarrow JR(X),
  \qquad
  \varepsilon_X\colon JC(X)\longrightarrow X.
\]
We call $\mathcal A$ bireflective, respectively bicoreflective, if it
is reflective, respectively coreflective, and every underlying map
$U(\eta_X)$, respectively $U(\varepsilon_X)$, is a bijection.

For $\Sigma\subseteq\mathbf I^{\mathbf I}$, a fuzzy topological space $(X,\tau)$
is called $\Sigma$-stable if $\sigma\circ\mu\in\tau$ for every
$\sigma\in\Sigma$ and every $\mu\in\tau$. We write $\operatorname{Stab}(\Sigma)$
for the full subconstruct of $\mathbf{FTS}$ whose objects are the
$\Sigma$-stable fuzzy topological spaces.

\begin{theorem}[Lowen--Wuyts {\cite[Theorem 2.11]{lowen1995}}]\label{thm:lowen}The map
  \[
    \tau\longmapsto\operatorname{Stab}(\tau)
  \]
  is a bijection from the total fuzzy topologies
  contained in $M$ onto the subconstructs of $\mathbf{FTS}$ that are both bireflective and bicoreflective.
\end{theorem}

\section{Proof of the main theorem}
\paragraph{The upper bound}
\begin{proposition}\label{upper}
  There exist at most $2^\mathfrak{c}$ total fuzzy topologies contained in $M$.
\end{proposition}
\begin{proof}
  For every $f\in M$ and $x\in(0,1]$, since $f$ is left-continuous, we have
  \[
    f(x)=\lim\limits_{\substack{q\to x^-\\q\in\mathbb{Q}}}f(q).
  \]Thus, $f$ is determined by its values on $\mathbb{Q}\cap[0,1]$. Hence, we have that \[| M|\leq\mathfrak{c}^{\aleph_0}=(2^{\aleph_0})^{\aleph_0}=2^{\aleph_0\cdot\aleph_0}=\mathfrak{c}.\]Consequently, the number of total fuzzy topologies contained in $ M$ is at most $|\mathcal{P}(M)|= 2^{|M|}\leq 2^{\mathfrak{c}}$.
\end{proof}

\subsection{The lower bound}

In the remainder of this section, we construct $2^{\mathfrak{c}}$ distinct total fuzzy topologies contained in $M$. Let $K$ be the set obtained by adjoining a point $*$ to $\mathbb{N}$, that is, $K=\mathbb{N}\cup\{*\}=\{0,1,2,\ldots,*\}$.

Set \[a_n=\frac12-2^{-n-1},\qquad
b_n=\frac12-\frac{3}{2^{n+3}}\]for $n\in\mathbb{N}$ and set $(a_*,b_*)=(3/4,7/8).$ These define a family of pairwise disjoint intervals. For example: \[(a_0,b_0)=(0,1/8),\quad(a_1,b_1)=(1/4,1/4+1/16),\quad(a_2,b_2)=(3/8,3/8+1/32),\ \dots\]

For every $A\subseteq K$, define $\rho_A:\mathbf{I}\to\mathbf{I}$ by
\[
  \rho_A(x)=
  \begin{cases}
    b_k,&\text{if }k\in A\text{ and }a_k<x\le b_k,\\
    x,&\text{otherwise}.
  \end{cases}
\]
\begin{lemma}\label{rhoalem}
  For all $A,B\subseteq K$ and every nonempty family $\{A_j\}_{j\in J}$:
  \begin{enumerate}[label=\normalfont(\arabic*)]
    \item $\rho_A\in M; $
    \item $\bigvee_{j\in J}\rho_{A_j}=\rho_{\bigcup_{j\in J}A_j}$;
    \item $\rho_A\wedge\rho_B=\rho_{A\cap B}$;
    \item $\rho_A\circ\rho_B=\rho_{A\cup B}$.

  \end{enumerate}
\end{lemma}
\begin{proof}
  Since $\rho_A$ is obtained by replacing the identity map on each $(a_k,b_k]$, $k\in A$, with the constant value $b_k,$ it is clear that $\rho_A$ is nondecreasing. For left-continuity, the only nontrivial point to check is $x=1/2.$ Since $\rho_A$ is nondecreasing, the left-continuity of $\rho_A$ at $1/2$ follows from \[\dfrac{1}{2}=\bigvee_{x<1/2} x\leq\bigvee_{x<1/2}\rho_A(x)\leq\dfrac{1}{2}=\rho_A\left(\dfrac{1}{2}\right). \]

  To prove (2)--(4), it suffices to consider points $x$ belonging to some $(a_k,b_k]$, since all the maps involved agree with the identity map elsewhere.

  Therefore,
  \[
    \bigvee_{j\in J}\rho_{A_j}(x)
    =
    \begin{cases}
      b_k,& k\in\bigcup_{j\in J}A_j,\\
      x,& k\notin\bigcup_{j\in J}A_j,
    \end{cases}
  \]
  which proves (2), and
  \[
    (\rho_A\wedge\rho_B)(x)
    =
    \begin{cases}
      b_k,& k\in A\cap B,\\
      x,& k\notin A\cap B,
    \end{cases}
  \]
  which proves (3).

  Finally, since $\rho_A(b_k)=b_k$ for every $A\subseteq K$, we have
  \[
    (\rho_A\circ\rho_B)(x)
    =
    \begin{cases}
      b_k,& k\in A\cup B,\\
      x,& k\notin A\cup B,
    \end{cases}
  \]
  which proves (4).
\end{proof}

Let $\mathcal{T}$ be a topology on $K$ and $\Gamma_\mathcal{T}$ be the fuzzy topology on $\mathbf{I}$ generated by $\{\rho_A\}_{A\in\mathcal{T}}.$
\begin{proposition}\label{prop:total}
  $\Gamma_\mathcal{T}$ is a total fuzzy topology contained in $M$.
\end{proposition}
\begin{proof}
  Since $ M$ is a fuzzy topology and the generators of $\Gamma_\mathcal{T}$ belong to $ M,$ every element of $\Gamma_\mathcal{T}$ belongs to $M$. Since $\rho_{\varnothing}=\mathrm{id},$ we have $\mathrm{id}\in\Gamma_\mathcal{T}.$

  It remains to prove that $\Gamma_\mathcal{T}$ is closed under composition. Fix some $A\in \mathcal{T}$, and let \[
    H_A=\{g\in  M\mid\rho_A\circ g\in\Gamma_\mathcal{T}\}.
  \]Given a nonempty family $\{f_i\}_{i\in I}\subseteq H_A$, by Lemma \ref{lem:left}, we have \[\rho_A\circ\left(\bigvee_{i\in I}f_i\right)=\bigvee_{i\in I}\left(\rho_A\circ f_i\right)\in\Gamma_\mathcal{T}.\]Monotonicity of $\rho_A$ shows that $H_A$ is closed under binary infima. For every $r\in\mathbf{I}$, it follows from $\rho_A\circ\underline{r}=\underline{\rho_A(r)}\in\Gamma_{\mathcal{T}}$ that all constant functions belong to $H_A.$ Thus, $H_A$ is a fuzzy topology.

  For every $B\in\mathcal{T},$ it follows from Lemma \ref{rhoalem} that $\rho_A\circ\rho_B=\rho_{A\cup B}\in\Gamma_\mathcal{T}$. Thus, the generators of $\Gamma_\mathcal{T}$ belong to $H_A$. Consequently, $\Gamma_\mathcal{T}\subseteq H_A$.

  Given $h\in \Gamma_\mathcal{T}$, let\[L_h=\{g\in  M\mid g\circ h\in\Gamma_\mathcal{T}\}.\] It is routine to check that $L_h$ is a fuzzy topology. For every $B\in\mathcal{T},$ since $h\in\Gamma_\mathcal{T}\subseteq H_B,$ it follows that $\rho_B\circ h\in\Gamma_\mathcal{T},$ and hence $\rho_B\in L_h$. Therefore, $\Gamma_\mathcal{T}\subseteq L_h$, which implies that $f\circ h\in \Gamma_\mathcal{T}$ for any $f\in\Gamma_\mathcal{T}.$ Since $h\in\Gamma_\mathcal{T}$ was arbitrary, we have that $\Gamma_\mathcal{T}$ is a total fuzzy topology.
\end{proof}

For every $k\in K$, choose numbers
\[
  a_k<x_k<t_k<b_k.
\]Define:
\begin{align*}
  \Pi\colon & M\to\mathcal{P}(K)\\
  &f\longmapsto\{k\in K\mid f(x_k)>t_k\}.
\end{align*}
\begin{lemma}
  \label{pilemma}
  For every family $(f_j)_{j\in J}\subseteq  M$, $f,g\in  M$ and $A\subseteq K$, we have
  \begin{enumerate}[label=\normalfont(\arabic*)]
    \item $\Pi\left(\bigvee_{j\in J}f_j\right)
      =\bigcup_{j\in J}\Pi(f_j)$;
    \item $\Pi(f\wedge g)=\Pi(f)\cap\Pi(g)$;
    \item $\Pi(\rho_A)=A$.
  \end{enumerate}
\end{lemma}
\begin{proof}
  If $J=\varnothing$, (1) is immediate. If $J\not=\varnothing$, then (1) follows from the fact that for every $k\in K$ we have\[ \bigvee_j f_j(x_k)>t_k
    \Longleftrightarrow
  f_i(x_k)>t_k\text{ for some }i\Longleftrightarrow k\in\bigcup_{j\in J}\Pi(f_j).\]

  Statements (2) and (3) are immediate.
\end{proof}
A topology $\mathcal{T}$ on $K$ is called constant-compatible if\[
  \Pi(\underline r)\in\mathcal{T}
  \quad\text{for every }r\in \mathbf{I}.
\]Equivalently, $\{0,1,\ldots,n\}\in\mathcal{T}$ for every $n\in\mathbb{N}$.
\begin{proposition}
  \label{prop:recovery}
  Let $\mathcal{T}$ be a constant-compatible topology on $K$. Then
  \[
    \Pi[\Gamma_\mathcal{T}]=\mathcal{T}.
  \]Consequently, the assignment $\mathcal{T}\mapsto\Gamma_\mathcal{T}$ is injective on the set of constant-compatible topologies.
\end{proposition}
\begin{proof}
  Let \[
    C_\mathcal{T}=\{f\in  M\mid\Pi(f)\in\mathcal{T}\}.
  \]Since $\mathcal{T}$ is constant-compatible, $\underline{r}\in C_\mathcal{T}$ for all $r\in\mathbf{I}$. Then, it follows from Lemma \ref{pilemma} that $C_\mathcal{T}$ is a fuzzy topology. Since $\rho_A\in C_\mathcal{T}$ for every $A\in\mathcal{T}$, all the generators of $\Gamma_\mathcal{T}$ belong to $C_\mathcal{T}$, and therefore $\Gamma_\mathcal{T}\subseteq C_\mathcal{T}.$ It follows that $\Pi[\Gamma_\mathcal{T}]\subseteq\mathcal{T}.$ Conversely, if $A\in\mathcal{T}$, then $\rho_A\in\Gamma_\mathcal{T}$ and $\Pi(\rho_A)=A$.  Thus
  $\mathcal{T}\subseteq\Pi[\Gamma_\mathcal{T}]$.
\end{proof}

Given an ultrafilter $\mathcal{U}$ on $\mathbb{N}$, define
\[ \mathcal{T}_{\mathcal U}
  =\mathcal{P}(\mathbb{N})
  \cup
\bigl\{\{*\}\cup U\mid U\in\mathcal U\bigr\}.\]
\begin{lemma}\label{lem:ultra}For every ultrafilter $\mathcal{U}$ on $\mathbb{N}$,
  \begin{enumerate}[label=\normalfont{(\arabic*)}]
    \item $\mathcal{T}_{\mathcal{U}}$ is a constant-compatible topology on $K$;
    \item the map $\mathcal U\mapsto\mathcal{T}_{\mathcal U}$ is injective.
  \end{enumerate}
\end{lemma}
\begin{proof}
  (1) follows from the filter axioms: unions are handled by the upward closure of $\mathcal U$, and finite
  intersections by its closure under finite intersections. Constant-compatibility is immediate.

  Since \[
    \mathcal U
    =
    \{U\subseteq\mathbb N\mid
    \{*\}\cup U\in\mathcal T_{\mathcal U}\},
  \]
  the map $\mathcal U\mapsto\mathcal T_{\mathcal U}$ is
  injective. This proves (2).
\end{proof}
We are now ready to prove the lower bound.
\begin{proposition}\label{lower}
  There exist at least $2^{\mathfrak{c}}$ total fuzzy topologies contained in $M$.
\end{proposition}
\begin{proof}
  By \cite{engelkingGeneralTopology1989}, there are $2^\mathfrak{c}$ ultrafilters on $\mathbb{N}$. The conclusion then follows from Lemma \ref{lem:ultra}, Propositions \ref{prop:recovery} and \ref{prop:total}.
\end{proof}
\begin{proof}[Proof of Theorem \ref{thm:main}]
  Propositions \ref{upper} and \ref{lower} show that there are exactly
  $2^{\mathfrak c}$ total fuzzy topologies contained in $M$.
  The Lowen--Wuyts bijection in Theorem \ref{thm:lowen} therefore yields exactly
  $2^{\mathfrak c}$ simultaneously bireflective and bicoreflective
  subconstructs of $\mathbf{FTS}$.
\end{proof}
\bibliographystyle{plain}
\bibliography{ref}
\end{document}